\documentclass[a4paper,12pt]{amsart}

\input xypic
\usepackage{amssymb}
\xyoption{all}

\newtheorem{thm}{Theorem}[section]
\newtheorem{lemma}[thm]{Lemma}

\newtheorem{prop}[thm]{Proposition}
\theoremstyle{definition}

\newtheorem{question}{Question}
\def\add{\operatorname{add}\nolimits}

\def\End{\operatorname{End}\nolimits}

\def\Hom{\operatorname{Hom}\nolimits}

\author{Jeremy Rickard}
\date\today
\title{Infinitely many algebras derived equivalent to a block}
\begin{document}
\maketitle
\section{Introduction}
Let $k$ be an algebraically closed field of characteristic $p>0$, let
$D$ be a finite $p$-group.

Brou\'e's Abelian Defect Conjecture \cite{Bro} predicts that if $D$ is
abelian and $G$ is a finite group with $D\leq G$, and $B$ is a block
algebra of $kG$ with defect group $D$, then $B$ is derived equivalent
to its Brauer correspondent, a block algebra of $N_G(D)$ which also
has defect group $D$.

Donovan's Conjecture predicts that there are, up to Morita
equivalence, only finitely many different block algebras of finite
groups with defect group $D$.

Both of these conjectures are still open, although progress has been
made on special cases of both. It is natural to ask (and a number of
people have asked) if they are related in the following way: if it
were true that there are only finitely many Morita equivalence classes
of algebras derived equivalent to each block of defect $D$, then
Brou\'e's Conjecture would imply Donovan's Conjecture for abelian $D$.

In some small cases, this is true. If $D$ is cyclic, every block with
defect $D$ is Morita equivalent to a Brauer tree algebra, and it is
known that the only algebras derived equivalent to a Brauer tree
algebra are also Morita equivalent to Brauer tree algebras for trees
with the same number of edges and same multiplicty of the exceptional
vertex, so there are only finitely many possibilities. If $D$ is a
Klein 4-group, then it is also known that there are only a very small
number of Morita classes of algebras derived equivalent to any block
with defect group $D$ (and in fact all of these occur as blocks of
group algebras.

However, for larger $D$ it seems that this is very rarely true. In
fact, we believe that the cases mentioned above are probably the only
cases where this is true.

In this paper, we give a method of showing by a fairly simple
calculation that a given block algebra has infinitely many Morita
equivalence classes of algebras derived equivalent to it, and show
that the method applies to several blocks with small abelian defect
group.

In fact, our method produces a sequence of algebras with unbounded
Cartan invariants (which is how we can detect that there are
infinitely many Morita equivalence classes. A weaker version of
Donovan's Conjecture states that the blocks with a given defect group
$D$ have bounded Cartan invariants, and this means that even this
weaker version wouldn't follow in a straightforward way from Brou\'e's
Conjecture.

\section{The main theorem}

We fix a field $k$, and a finite-dimensional algebra $A$ over $k$.

By a ``module'' for such an algebra, we shall always mean a right
module unless specified otherwise.

If $X$ is an object of an additive category, then by $\add(X)$ we mean
the category of direct summands of finite direct sums of copies of $X$. 

Let $A$ be a $k$-algebra. By $\mod(A)$ we denote the category of
finitely generated $A$-modules, and by $P_A$ the category of finitely
generated $A$-modules. So, if $\!_AA$ is the regular $A$-module,
considered as an object of $\mod(A)$, then $P_A$ is just
$\add(\!_AA)$. We also identify $\mod(A)$ in the usual way with the
full subcategory of $D^b(A)$ consisting of complexes concentrated in
degree zero.
 
By $K^b(P_A)$ we denote the homotopy category of bounded complexes
over $P_A$, and by $D^b(A)$ the bounded derived category of
$A$-modules. We regard $K^b(P_A)$ as a full subcategory of $D^b(A)$ in
the usual way.

Let $\{S_i:i\in I\}$ be a set of representatives of the isomorphism
classes of simple $A$-modules, and for each $i\in I$ let $P_i$ be a
projective cover of $S_i$.

Recall \cite{Ri1} that a {\bf tilting complex} for $A$ is an object $T$ of
$K^b(P_A)$ such that 
$$
\Hom_{K^b(P_A)}\left(T,T[n]\right)=0
$$
if $n\neq0$ and such that $\add(T)$ generates $K^b(P_A)$ as a
triangulated category, and that a $k$-algebra $B$ is derived
equivalent to $A$ if and only if $B\cong\End_{K^b(P_A)}(T)$ for $T$
some tilting complex for $A$.

The following construction of a tilting complex for a symmetric
$k$-algebra was described in \cite{Ri2} and \cite{Oku}, but since
neither of these has been published (at least, not in English), we
shall give details here.

First, recall \cite[Corollary 3.2]{Ri4} the following duality for 
finite-dimensional symmetric $k$-algebras.

\begin{lemma}\label{dual} 
  Let $A$ be a finite-dimensional symmetric $k$-algebra, and $X$, $Y$
  objects of $K^b(P_A)$. Then $\Hom_{K^b(P_A)}(X,Y)$ is dual to
  $\Hom_{K^b(P_A)}(X,Y)$, and hence (applying this with $Y[n]$ in
  place of $Y$), $\Hom_{K^b(P_A)}(X,Y[n])$ is dual to
  $\Hom_{K^b(P_A)}(Y,X[-n])$.
\end{lemma}
 
Let $I_0$ be a subset of the set $I$ indexing the simple $A$-modules,
and let $i\in I$.

If $i\in I_0$, let $T_i$ be the object 
$$
\dots\rightarrow0\rightarrow P_i\rightarrow 0\rightarrow0\rightarrow\dots
$$
of $K^b(P_A)$, where $P_i$ is in degree 1.

If $i\in I\setminus I_0$, then let $T_i$ be the object
$$
\dots0\rightarrow R_i\rightarrow P_i\rightarrow 0
$$
of $K^b(P_A)$, where $P_i$ is in degree zero, and $R_i\rightarrow
P_i\rightarrow M_i\rightarrow 0$ is the minimal projective
presentation of the largest quotient $M_i$ of $P_i$ with no
composition factors isomorphic to elements of $\{S_j:j\in I_0\}$ (or
equivalently, $R_i$ is the projective cover of the largest submodule
of $P_i$ whose simple quotients are all isomorphic to elements of
$\{S_j:j\in I_0\}$). Thus, in particular, $R_i$ is a direct sum of
indecomposable projective modules from $\{P_j:j\in I_0\}$ and every
map $P_j\to P_i$, with $j\in I_0$, factors through $R_i\to P_i$.

\begin{prop}
If $A$ is a finite-dimensional symmetric $k$-algebra, then 
$$
T=\bigoplus_{i\in I}T_i
$$
is a tilting complex for $A$.
\end{prop}

\begin{proof}
  Clearly $\Hom_{K^b(P_A)}(T_i,T_j[n])=0$ if $i,j\in I_0$ and
  $n\neq0$, since then there are no chain maps $T_i\to T_j[n]$, as
  $T_i$ is concentrated in degree 1, but $T_j$ is concentrated in
  degree $n+1$.

  For similar reasons, there are no chain maps $T_i\to T_j[n]$ if
  $i\in I_0$ and $j\in I\setminus I_0$ except possibly if $n=0$ or
  $n=1$. But then, since every map $P_i\to P_j$ factors through
  $R_j\to P_j$, every chain map $T_i\to T_j[1]$ is
  null-homotopic. Hence $\Hom_{K^b(P_A)}(T_i,T_j[n])=0$ for $n\neq0$,
  and by duality (Lemma~\ref{dual}), $\Hom_{K^b(P_A)}(T_j,T_i[n])=0$ for
  $n\neq0$

  Similarly, if $i,j\in I\setminus I_0$, there are no chain maps
  $T_i\to T_j[n]$ except possibly for $n\in\{-1,0,1\}$, since $T_i$
  and $T_j$ are concentrated in degrees 1 and 0. But since every map
  $R_i\to P_j$ factors through $R_j\to P_j$, every chain map $T_i\to
  T_j$ is null-homotopic, so $\Hom_{K^b(P_A)}(T_i,T_j[1])=0$ and, by
  duality, $\Hom_{K^b(P_A)}(T_j,T_i[-1])=0$.

  This shows that $\Hom_{K^b(P_A)}(T,T[n])=0$ for $n\neq0$.

  Finally, to show that $\add(T)$ generates $P_A$ as a triangulated
  category, it suffices to note that if $i\in I_0$ then $P_i[1]=T_i$
  is in $\add(T)$, and if $i\in I\setminus I_0$ then the triangle
$$R_i\to P_i\to T_i\to R_i[1]$$
shows that $P_i$ is in the triangulated category generated by
$\add(T)$, since $R_i[1]$ and $T_i$ are in $\add(T)$.
\end{proof}

Let $B=\End_{K^b(P_A)}(T)$, so that there is an equivalence of derived
categories between $F:D^b(A)\to D^b(B)$ sending the objects $T_i$ to
the indecomposable projective $B$-modules. In this way the isomorphism
classes of indecomposable projective $B$-modules are naturally indexed
by $I$, with the projective indexed by $i\in I_0$ being $F(P_i)[1]$.

Since the class of finite-dimensional symmetric $k$-algebras is closed
under derived equivalence~\cite[Corollary 5.3]{Ri3} we can iterate
this construction, using the same subset $I_0\subset I$, obtaining
after $t$ iterations a tilting complex 
$$
T^{(t)}=\bigoplus_{i\in I}T^{(t)}_i
$$ 
for $A$, whose
indecomposable summands are $T^{(t)}_i=P_i[t]$ for $i\in I_0$ and
$$
T^{(t)}_i=\dots\to0\to R_i^{(t-1)}\to R_i^{(t-2)}\to\dots\to
R_i^{(0)}\to P_i\to0\to\dots
$$ 
(with $P_i$ in degree zero) for $i\in I\setminus I_0$, where
$R_i^{(0)}\to P_i$ is the map $R_i\to P_i$ defined above and for
$k>0$, $R_i^{(k)}$ is the projective cover of the largest submodule of
the kernel of the next differential whose simple quotients are all
isomorphic to elements of $\{S_i:i\in I_0\}$.

Thus, if we set
$$
B^{(t)}=\End_{K^b(P_A)}(T^{(t)}),
$$
we obtain a sequence $B=B^{(0)},B^{(1)},\dots$ of algebras that are
derived equivalent to $A$. For $i\in I$, let $P^{(t)}_i$ be the
indecomposable projective $B^{(t)}$-module corresponding to
$T^{(t)}_i$ under the derived equivalence between $B^{(t)}$ and $A$.

We shall give a technique that allows us to prove by a simple
calculation that the sequence $B^{(0)},B^{(1)},\dots$ of algebras very
often contains infinitely many algebras from different Morita
equivalence classes, even when $A$ is a block of a finite group
algebra, and therefore, in such cases, if Donovan's Conjecture is
true, almost all the algebras in the sequence are not Morita
equivalent to blocks of finite group algebras.
  
Let us start by giving another interpretation of this construction.

Let 
$$
Q_0=\bigoplus_{i\in I_0}P_i,
$$
and $E=\End_A(Q_0)$. Then, considering $Q_0$ as an $E$-$A$-bimodule,
the functor 
$$
H=\Hom_A\left(Q_0,-\right):\mod(A)\to\mod(E)
$$ 
restricts to an equivalence of categories
$$
\add(Q_0)\to P_E.
$$

Applying this functor to the complex $T^{(t)}_i$ for
$i\in I\setminus I_0$, we obtain a complex
$$\dots\to0\to H\left(R_i^{(t-1)}\right)\to 
H\left(R_i^{(t-2)}\right)\to\dots\to 
H\left(R_i^{(0)}\right)\to 
H\left(P_i\right)\to0\to\dots
$$
where, since $R^{(k)}_i$ is in $\add(Q_0)$ for each $k$,
$H\left(R^{(k)}\right)$ is a projective $E$-module, and the
construction of the differentials in $T^{(t)}_i$ translates into the fact
that when we apply $H$ we get a complex that, being acyclic except at
$H\left(R^{(t-1)}_i\right)$, is the truncation of a minimal projective
$E$-module resolution of $H\left(P_i\right)$.

Now, consider the Cartan invariants of $B^{(t)}$,
$$c^{(t)}_{ij}=\dim_k\Hom_{B^{(t)}}\left(P^{(t)}_i,P^{(t)}_j\right).$$

Then translating to the derived category of $A$,
$$c^{(t)}_{ij}=\dim_k\Hom_{K^b(P_A)}\left(T^{(t)}_i,T^{(t)}_j\right).$$

If we take $i\in I_0$ but $j\in I\setminus I_0$, so that $T^{(t)}_i=P_i[t]$ and 
$$T^{(t)}_j=\dots\to0\to R_j^{(t-1)}\to R_j^{(t-2)}\to\dots\to
R_j^{(0)}\to P_j\to0\to\dots,
$$ 
then there is a short exact sequence of complexes
$$
0\to T^{(t)}_j\to T^{(t+1)}_j\to R_j^{(t)}[t+1]\to 0, 
$$
and since $\Hom_{K^b(P_A)}\left(T^{(t)}_i,T^{(t)}_j[n]\right)=0$ for
$n\neq0$, $\Hom_{K^b(P_A)}\left(T^{(t)}_i,T^{(t+1)}_j[n]\right)=0$ for
$n\neq-1$ (since $T^{(t+1)}_i=T^{(t)}_i[1]$) and
$\Hom_{K^b(P_A)}\left(T^{(t)}_i,R_j^{(t)}[n]\right)=0$ for $n\neq
t$, the long exact sequence obtained by applying the cohomological
functor $\Hom_{K^b(P_A)}\left(T^{(t)}_i,-\right)$ to this sequence
yields a short exact sequence
$$
0\to\Hom_{K^b(P_A)}\left(T^{(t+1)}_i,T^{(t+1)}_j\right)
\to\Hom_{K^b(P_A)}\left(T^{(t)}_i,R_j^{(t)}[t]\right)
\to\Hom_{K^b(P_A)}\left(T^{(t)}_i,T_j^{(t)}\right)
\to0
$$
and hence
$$
c^{(t)}_{ij}+c^{(t+1)}_{ij}=\dim_k\Hom_A\left(P_i,R_j^{(t)}\right).
$$

Therefore, if 
$$
\left\{\dim_k\Hom_A\left(P_i,R_j^{(t)}\right): 0\leq t<\infty\right\}
$$
is unbounded, then the set of Cartan invariants
$\left\{c^{t)}_{ij}:0\leq t<\infty\right\}$ must also be unbounded,
and hence the sequence of algebras $B^{(0)},B^{(1)},\dots$ must
contain representatives of infinitely many Morita equivalence classes.

Since we have seen that the projective modules $R_j^{(t)}$ correspond
under an equivalence of categories to the terms in a minimal
projective resolution of the $E$-module $\Hom_A(Q_0,P_j)$, where
$E=\End_A(Q_0)$, and the modules $P_i$ for $i\in I_0$ correspond to
the indecomposable projective $E$-modules under the same equivalence
of categories, the following theorem follows.

\begin{thm}\label{main}
  Let $A$ be a symmetric $k$-algebra, let $\{S_i:i\in I\}$ be a set of
  representatives of the isomorphism classes of simple modules, and
  for each $i\in I$ let $P_i$ be a projective cover of $S_i$. Let
  $I_0$ be a subset of $I$, $Q_0=\bigoplus_{i\in I_0}P_i$ and
  $E=\End_A(Q_0)$. If there is $j\in I\setminus I_0$ such that
  $\Hom_A(Q_0,P_j)$, considered as an $E$-module, has a minimal
  projective resolution whose terms have unbounded dimension, then
  there are infinitely many Morita equivalence classes of algebras
  derived equivalent to $A$.
\end{thm}

\section{Applications}

We shall apply Theorem~\ref{main} to some blocks of group algebras,
but first let us explain why it doesn't apply in the small cases where
there are known to be only finitely many Morita equivalence classes of
algebras derived equivalent to a certain block.

In the case of blocks $A$ with cyclic defect group, then whatever
subset $I_0\subset I$ we choose, the endomorphism algebra
$\End_A\left(Q_0\right)$ is periodic, and so the minimal projective
resolution of any $E$-module has bounded terms, so our theorem doesn't
apply.

Similarly, let us take $A$ to be a block with defect group $C_2\times
C_2$: let us take $A$ to be the group algebra of the alternating group
$A_4$, although similar remarks apply to any block with the same defect
group. There are three simple modules, so we have the choice of taking
$|I_0|=1$ or $|I_0|=2$ (it is easy to see that if $I_0=\emptyset$ or
$I_0=I$, the algebras $B^{(t)}$ that we construct are all Morita
equivalent to $A$). If $|I_0|=1$, then $E$ is isomorphic to the
algebra $k[x]/(x^2)$, which is periodic. If $|I_0|=2$, then $E$ is a
Brauer tree algebra for the tree with two edges, and so again is
periodic. So in neither case are there any $E$-modules with unbounded
minimal projective resolution.

Another obvious case where our theorem certainly can't be applied is
to a block with only one simple module (so the only possibilities are
$I_0=\emptyset$ or $I_0=I$. Again, in this case, there are not
infinitely many Morita classes of algebras derived equivalent to the
block, since for a local algebra, the only tilting complexes are
isomorphic to shifts of projective generators, and so all derived
equivalent algebras are in fact Morita equivalent.

It may well be that the cases we have just described are the only
examples of blocks with abelian defect group which are derived
equivalent to only finitely many Morita equivalence classes of blocks,
although there are some more small cases where Theorem~\ref{main} does
not prove this.

Now let us examine some cases where the theorem does apply.

In characteristic 3, let $D=C_3\times C_3$ be an elementary abelian
group of rank 2.

First let $G$ be the semidirect product $D\rtimes C_2$, where a
generator of $C_2$ acts on $D$ by inverting all elements, and take
$A=kG$.

Then $A$ has two simple modules, obtained by inflating the two
one-dimensional modules for $kC_2$. Denote these by $S_k$ and
$S_{\epsilon}$. We'll take $I_0=\{k\}$, so that $Q_0=P_k$.

The projective modules $P_k$ and $P_{\epsilon}$ have Loewy series that are respectively
$$
\begin{array}{ccccc}
&&S_k&&\\
&S_{\epsilon}&&S_{\epsilon}&\\
S_k&&S_k&&S_k\\
&S_{\epsilon}&&S_{\epsilon}&\\
&&S_k&&
\end{array}
$$
and
$$
\begin{array}{ccccc}
&&S_{\epsilon}&&\\
&S_k&&S_k&\\
S_{\epsilon}&&S_{\epsilon}&&S_{\epsilon}\\
&S_k&&S_k&\\
&&S_{\epsilon}&&
\end{array}
$$

The endomorphism algebra $E$ of $P_k$ is a 5-dimensional commutative
local algebra, with Loewy length 3, generated by three elements $x$,$y$ and $z$ subject to
the relations $x^2=xy=yz=z^2=0$ and $xz=y^2$. The set
$\{1,x,y,z,y^2\}$ is a basis. 

The $M=\Hom_A\left(P_k,P_{\epsilon}\right)$ is a 4-dimensional
indecomposable $E$-module, with Loewy length 2, generated by two
elements $u$ and $v$, with $ux=uy^2=vy^2=vz=0$. A basis is given by
$\{u,v,uz,vx\}$. $M$ has two-dimensional head and two-dimensional
socle.

Since $E$ is a symmetric algebra, the syzygies $\Omega^sM$ will all be
indecomposable with no projective summands, and so have Loewy length
at most two, and the socle of $\Omega^{s+1}M$ will be isomorphic to
the head of $\Omega^sM$ for every $s$. Denoting the dimension of the socle of $\Omega^sM$ by $a_s$ (so the dimension of its head is $a_{s+1}$), the projective cover of $\Omega^sM$ will be the direct sum of $a_{s+1}$ copies of the regular $E$-module. Since the dimension of $E$ is 5, the short exact sequence
$$0\to\Omega^{s+1}M\to E^{a_{s+1}}\to\Omega^sM\to 0$$
gives a recurrence relation
$$a_{s+1}=3a_s+1.$$
We have $a_0=2=a_1$, and solving the recurrence relation we have
$$a_s=\left(1-\frac{1}{\sqrt{5}}\right)\left(\frac{3+\sqrt{5}}{2}\right)^s +
\left(1+\frac{1}{\sqrt{5}}\right)\left(\frac{3-\sqrt{5}}{2}\right)^s,$$
and in particular $a_s$ grows exponentially, so Theorem~\ref{main}
tells us that there are infinitely many Morita equivalence classes of
algebras derived equivalent to $kG$.

However, if we take a different semidirect product $G=D\rtimes C_2$,
where now a generator of $C_2$ acts trivially on one cyclic factor of
$D=C_3\times C_3$, but by inversion on the other, then the
indecomposable projective $kG$-modules have Loewy series
$$
\begin{array}{ccccc}
&&S_k&&\\
&S_k&&S_{\epsilon}&\\
S_k&&S_{\epsilon}&&S_k\\
&S_{\epsilon}&&S_k&\\
&&S_{\epsilon}&&
\end{array}
$$
and
$$
\begin{array}{ccccc}
&&S_{\epsilon}&&\\
&S_{\epsilon}&&S_k&\\
S_{\epsilon}&&S_k&&S_{\epsilon}\\
&S_k&&S_{\epsilon}&\\
&&S_{\epsilon}&&
\end{array}
$$

This time the endomorphism algebra $E=\End_{kG}\left(P_k\right)$ is a
six-dimensional commutative algebra, and
$M=\Hom_{kG}\left(P_k,P_{\epsilon}\right)$ is a three-dimensional
$E$-module with $\Omega M\cong M$, so the minimal projective
resolution of $M$ is periodic, and our main theorem doe not apply.

Of course, this doesn't prove that there are only finitely many Morita
equivalence classes of algebras that are derived equivalent to $kG$,
and we don't know whether or not this is in fact the case.

\section{Questions and concluding remarks}

Our construction gives examples of families of infinitely many derived
equivalent algebras with unbounded Cartan invariants. But there are
well-known conjectures in modular representation theory (such as the
weaker form of Donovan's Conjecture, or some more precise conjectures)
which predict a bound on the Cartan invariants of a block in terms of
the defect group.

So an obvious question is whether, if such a bound were proved, we
could then deduce Donovan's Conjecture (at least for blocks with
abelian defect group) from Brou\'e's Abelian Defect Group
Conjecture. In other words:

\begin{question}
  Are there infinite families of Morita equivalence classes of
  algebras, all derived equivalent to the same block algebra, but with
  bounded Cartan invariants.
\end{question}

It is known that there is only a countable set of algebras derived
equivalent to a given algebra, so there can not be continuous families
of derived equivalent algebras, at least over an uncountable field,
but it does not seem obvious that there cannot be an infinite discrete
set of derived equivalent, but not Morita equivalent, algebras with
the same Cartan matrix. We don't know any examples of this kind.

Another question that arises is that the construction proving our main
theorem provides examples of infinite families of algebras derived
equivalent to a given block, but if even the weaker form of Donovan's
Conjecture is true, only finitely many of them can be Morita
equivalent to block algebras. However, they have most of the algebraic
properties that they would have to have to be block algebras, by
virtue of being derived equivalent to one, since many obvious
properties are preserved by derived equivalence. 

Some are ruled out by partial results in the direction of Donovan's
Conjecture, but to the best of our knowledge, all these results rely
on the Classification of Finite Simple Groups, or at least, most of
the force of that theorem..

\begin{question}
  Is there any way to prove that any of the algebras we construct are
  {\bf not} Morita equivalent to any block algebras, short of using
  the Classification of Finite Simple Groups or something close to that?
\end{question}

\end{document}